\documentclass[reqno,11pt]{article}
\usepackage{a4wide}
\usepackage{color,eucal,enumerate,mathrsfs,amsthm}
\usepackage[normalem]{ulem}
\usepackage{amsmath,amssymb,amsthm,epsfig,bbm,makecell}
\numberwithin{equation}{section}

\usepackage{amsfonts}
\usepackage{dsfont}
\usepackage{mathtools}
\usepackage{leftidx}
\usepackage{graphicx}
\usepackage{esint}
\usepackage{authblk}
\usepackage{multirow}

\usepackage[pdfborder={0 0 0}]{hyperref}  

\usepackage[latin1]{inputenc}
\usepackage[english]{babel}

\newtheorem{Definition}{Definition}[section]
\newtheorem{Proposition}[Definition]{Proposition}
\newtheorem{Lemma}[Definition]{Lemma}
\newtheorem{Theorem}[Definition]{Theorem}
\newtheorem{Corollary}[Definition]{Corollary}
\theoremstyle{definition}

\newtheorem{Remark}[Definition]{Remark}
\newtheorem{Setting}[Definition]{Setting}

\allowdisplaybreaks

\newcommand{\N}{\mathbb{N}}

\newcommand{\R}{\mathbb{R}}

\newcommand{\mm}{{\mbox{\boldmath$m$}}}

\newcommand{\sfd}{{\sf d}}

\newcommand{\sfh}{{\sf h}}

\newcommand{\Kliminf}{K\kern-3pt-\kern-2pt\mathop{\rm lim\,inf}\limits}  
\newcommand{\supp}{\mathop{\rm supp}\nolimits}   
\newcommand{\Lip}{\mathop{\rm Lip}\nolimits}          
\renewcommand{\d}{{\mathrm d}}
\newcommand{\dt}{{\d t}}
\newcommand{\ddt}{{\frac \d\dt}}

\newcommand{\restr}[1]{\lower3pt\hbox{$|_{#1}$}} 

\newcommand{\eps}{\varepsilon}  
\newcommand{\nchi}{{\raise.3ex\hbox{$\chi$}}}

\newcommand{\fr}{\penalty-20\null\hfill$\blacksquare$}                      

\newcommand{\prob}[1]{\mathscr P(#1)}                   
\newcommand{\probt}[1]{\mathscr P_2(#1)}                   

\renewcommand{\mm}{\mathfrak m}                                

\renewenvironment{proof}{\removelastskip\par\medskip   
\noindent{\em proof} \rm}{\penalty-20\null\hfill$\square$\par\medbreak}

\newcommand{\X}{{\rm X}}

\newcommand{\h}{{\sfh}}

\renewcommand{\ae}{{\textrm{\rm{-a.e.}}}}

\newcommand{\CD}{{\sf CD}}
\newcommand{\RCD}{{\sf RCD}}

\newcommand{\lip}{{\rm lip}}

\newcommand{\HS}{{\lower.3ex\hbox{\scriptsize{\sf HS}}}}

\newcommand{\hr}{{\sf r}}

\title{From Harnack inequality to heat kernel estimates on metric measure spaces and applications}
\author{Luca Tamanini \thanks{Institut f\"ur Angewandte Mathematik, Universit\"at Bonn. email: tamanini@iam.uni-bonn.de}}

\begin{document}

\maketitle

\begin{abstract}
Aim of this short note is to show that a dimension-free Harnack inequality on an infinitesimally Hilbertian metric measure space where the heat semigroup admits an integral representation in terms of a kernel is sufficient to deduce a sharp upper Gaussian estimate for such kernel. As intermediate step, we prove the local logarithmic Sobolev inequality (known to be equivalent to a lower bound on the Ricci curvature tensor in smooth Riemannian manifolds). Both results are new also in the more regular framework of $\RCD(K,\infty)$ spaces.
\end{abstract}

\tableofcontents

\section{Introduction}

In a smooth Riemannian manifold $(M,g,{\rm vol})$ with Ricci curvature bounded from below it is well-known \cite{Sturm92} that the heat kernel $\hr_t$, namely the positive fundamental solution of the heat equation defined by the Laplace-Beltrami operator (or, from a probabilistic viewpoint, the transition probability of the Brownian motion), satisfies two-sided Gaussian estimates, which read as
\begin{equation}
\label{eq:gaussest}
\frac{1}{C_1{\rm vol}(B_{\sqrt t}(y))}\exp\Big(-\frac{\sfd^2(x,y)}{(4-\eps)t} - C_2 t\Big)\leq \hr_t[x](y)\leq \frac{C_1}{{\rm vol}(B_{\sqrt t}(y))}\exp\Big(-\frac{\sfd^2(x,y)}{(4+\eps)t} + C_2 t\Big)
\end{equation}
for every $\eps>0$, $x,y\in \X$ and $t>0$, for suitable positive constants $C_1=C_1(K,N,\delta)$ and $C_2=C_2(K,N,\delta)$, where $K$ is a lower bound on the Ricci curvature and $N$ an upper bound for the dimension. But because of the great interest in the study of metric (measure) spaces (we refer to \cite{Heinonen07} for an overview on the topic and detailed bibliography) and the essentially metric-measure nature of \eqref{eq:gaussest}, it is rather natural to investigate possible generalizations of these kind of bounds to more abstract settings. For regular, symmetric, strongly local Dirichlet forms on locally compact separable Hausdorff spaces satisfying doubling \& Poincar\'e, this has been achieved by K.-T.\ Sturm in \cite{Sturm96II}, while more recently R.\ Jiang, H.\ Li and H.\ Zhang studied the problem on finite-dimensional $\RCD^*(K,N)$ spaces \cite{JLZ15}, which are still locally compact but only locally doubling for $K<0$. This class of metric measure spaces with Ricci curvature bounded from below, introduced in \cite{AmbrosioGigliSavare11-2}, \cite{Gigli12} starting from the seminal papers \cite{Lott-Villani09}, \cite{Sturm06I}, \cite{Sturm06II}, is the natural framework where bounds like \eqref{eq:gaussest} can be expected to hold, since it enjoys many analytical and geometric properties and the existence of the heat kernel is well understood.

\medskip

As concerns weighted Riemannian manifolds $(M,g,\mu)$, the picture is not so clear, as to the best of our knowledge no lower bounds are known (unless $M$ has finite volume, see \cite{Wang11}), whereas an upper estimate can still be deduced \cite{Grigoryan94b}, \cite{GongWang01} and reads as
\begin{equation}\label{eq:gaussest2}
\hr_t[x](y) \leq \frac{1}{\sqrt{\mu(B_{\sqrt{t}}(x))\mu(B_{\sqrt{t}}(y))}}\exp\Big(C_\eps(1+C_K t) - \frac{\sfd^2(x,y)}{(4+\eps)t}\Big)
\end{equation}
for every $\eps>0$, $x,y\in \X$ and $t>0$; notice however that in general it is not possible to get rid of one of the two volumes on the right-hand side, since $M$ does not need to be doubling. As in the previous discussion, also \eqref{eq:gaussest2} has an essentially metric-measure nature and if on the one hand the non-smooth counterpart of smooth Riemannian manifolds with lower Ricci bounds is given by $\RCD^*(K,N)$ spaces, on the other hand $\RCD(K,\infty)$ spaces represent the natural alternative to weighted manifolds and in this setting \eqref{eq:gaussest2} is still missing; therefore it would be natural, as a first attempt, to prove \eqref{eq:gaussest2} on any $\RCD(K,\infty)$ space. 

Yet, after a careful look at \cite{JLZ15} one can observe that the key role in the proof of \eqref{eq:gaussest} is played by the dimension-dependent Harnack inequality \cite{GarofaloMondino14}, \cite{Jiang15} and that not surprisingly in \cite{GongWang01} the dimension-free Harnack inequality, which is known to hold on $\RCD(K,\infty)$ spaces after \cite{Li16}, is a crucial ingredient as well. For this reason, instead of establishing \eqref{eq:gaussest2} on an $\RCD(K,\infty)$ space, we prefer to work in a more general environment, namely a metric measure space $(\X,\sfd,\mm)$ supporting a dimension-free Harnack inequality where the heat semigroup admits an integral representation (see Setting \ref{set} below). In order to achieve this goal we obtain a local logarithmic Sobolev inequality, that to the best of our knowledge was missing also in the context of $\RCD(K,\infty)$ spaces, and the $L^\infty$-${\rm LIP}$ regularization for the heat flow (see Proposition \ref{pro:loclogsob}).

\medskip

The paper is structured as follows. In Section \ref{sec:2} we recall all the relevant notions, results and bibliographical references related to calculus on metric measure spaces and point out the precise framework we shall work within. In Section \ref{sec:3} we collect some auxiliary results, most notably the local logarithmic Sobolev inequality. Finally, Section \ref{sec:4} is devoted to the proof of the upper Gaussian estimate for the heat kernel and some consequences.

\bigskip

\noindent{\bf Acknowledgements} 

The author gratefully acknowledges support by the European Union through the ERC-AdG ``RicciBounds'' for Prof.\ K.\ T.\ Sturm and would like to thank Prof.\ M.\ Gordina for useful suggestions.

\section{Preliminaries and setting}\label{sec:2}

By ${\rm LIP}(\X)$ we denote the space of Lipschitz continuous functions and by $C([0,1],\X)$ the space of continuous curves with values in the metric space $(\X,\sfd)$. For the notion of absolutely continuous curve in a metric space and of metric speed see for instance Section 1.1 in \cite{AmbrosioGigliSavare08}. The collection of absolutely continuous curves on $[0,1]$ is denoted by $AC([0,1],\X)$. By $\prob\X$ we denote the space of Borel probability measures on $(\X,\sfd)$ and by $\probt\X \subset \prob\X$ the subclass of those with finite second moment.

\medskip

Let $(\X,\sfd,\mm)$ be a complete and separable metric measure space endowed with a Borel non-negative measure which is finite on bounded sets. 

For the definition of the {\bf Sobolev class} $S^2(\X)$ and of {\bf minimal weak upper gradient} $|D f|$ see \cite{AmbrosioGigliSavare11} (and the previous works \cite{Cheeger00}, \cite{Shanmugalingam00} for alternative - but equivalent - definitions of Sobolev functions). The Sobolev space $W^{1,2}(\X)$ is defined as $L^2(\X)\cap S^2(\X)$. When endowed with the norm $\|f\|_{W^{1,2}}^2:=\|f\|_{L^2}^2+\||Df|\|_{L^2}^2$, $W^{1,2}(\X)$ is a Banach space. The {\bf Cheeger energy} is the convex and lower-semicontinuous functional $E:L^2(\X)\to[0,\infty]$ given by
\[
E(f):=\left\{\begin{array}{ll}
\displaystyle{\frac12\int|D f|^2\,\d\mm}&\qquad \text{for }f\in W^{1,2}(\X)\\
+\infty&\qquad\text{otherwise}
\end{array}\right.
\]
$(\X,\sfd,\mm)$ is {\bf infinitesimally Hilbertian} (see \cite{Gigli12}) if $W^{1,2}(\X)$ is Hilbert. In this case the {\bf cotangent module} $L^2(T^*\X)$ (see \cite{Gigli14}) and its dual, the {\bf tangent module} $L^2(T\X)$, are canonically isomorphic, the {\bf differential} is a well-defined linear map $\d$ from $S^2(\X)$ with values in $L^2(T^*\X)$ and the isomorphism sends the differential $\d f$ to the gradient $\nabla f$. Furthermore $E$ is a Dirichlet form admitting a carr\'e du champ given by $\langle\nabla f,\nabla g\rangle$, where $\langle\cdot,\cdot\rangle$ is the pointwise scalar product on the Hilbert module $L^2(T\X)$. The infinitesimal generator $\Delta$ of $E$, which is a closed self-adjoint linear operator on $L^2(\X)$, is called {\bf Laplacian} on $(\X,\sfd,\mm)$ and its domain denoted by $D(\Delta)\subset W^{1,2}(\X)$. A function $f \in W^{1,2}(\X)$ belongs to $D(\Delta)$ and $g = \Delta f$ if and only if
\[
\int\phi g\,\d\mm = -\int\langle\nabla\phi,\nabla f\rangle\,\d\mm, \qquad \forall \phi \in W^{1,2}(\X).
\]
The flow $(\h_t)$ associated to $E$ is called {\bf heat flow} (see \cite{AmbrosioGigliSavare11}), and for any $f\in L^2(\X)$ the curve $t\mapsto\h_tf\in L^2(\X)$ is continuous on $[0,\infty)$, locally absolutely continuous on $(0,\infty)$ and the only solution of
\[
\ddt\h_tf = \Delta\h_tf, \qquad \h_tf \to f\text{ as }t \downarrow 0.
\]
After this preliminary part, we can describe the framework we shall work within throughout this note.

\begin{Setting}\label{set}
$(\X,\sfd,\mm)$ is a complete and separable metric space equipped with a non-negative Borel measure which is finite on bounded sets and supports the following dimension-free Harnack inequality: for any $p \in (1,\infty)$, $x,y \in \X$ and for any $f \in L^1 \cap L^\infty(\X)$ it holds
\begin{equation}\label{eq:harnack}
|\h_t f(x)|^p \leq (\h_t|f|^p)(y)\exp\Big(\frac{pK\sfd^2(x,y)}{2(p-1)(e^{2Kt}-1)}\Big)
\end{equation}
with $K \in \R$. We also assume that there exists a function
\begin{equation}
\label{eq:hk}
(0,\infty)\times \X^2\ni (t,x,y)\quad\mapsto\quad \hr_t[x](y) \in (0,\infty)
\end{equation}
called heat kernel, such that for every $x \in \X$ and $t>0$, $\hr_t[x]$ is a probability density, and the following identity holds
\begin{equation}
\label{eq:rapprform}
\h_tf(x) = \int f(y)\hr_t[x](y)\,\d\mm(y) \qquad \forall t>0,\, \forall f \in L^2(\X).
\end{equation}
\end{Setting}

By the arguments in \cite{AmbrosioGigliSavare11-2} and \cite{AmbrosioGigliMondinoRajala15} with slight adaptations, the semigroup property of $\h_t$ and the representation formula \eqref{eq:rapprform} entail that the heat kernel is symmetric, i.e.\ $\hr_t[x](y) = \hr_t[y](x)$ $\mm \otimes \mm$-a.e.\ in $\X^2$, and satisfies the Chapman-Kolmogorov formula
\begin{equation}
\label{eq:chapkol}
\hr_{t+s}[x](y) = \int \hr_t[x](z)\hr_s[z](y)\d\mm(z) \qquad \textrm{for }\mm \otimes \mm\textrm{-a.e. }(x,y) \in \X^2,\,\forall t,s \geq 0.
\end{equation}
Moreover, \eqref{eq:rapprform} and the fact that $\hr_t[x]$ is a probability density can also be used to extend the heat flow to $L^1(\X)$, show that it is {\bf mass preserving} and satisfies the {\bf maximum principle}, i.e.
\[
f\leq c\quad\mm-a.e.\qquad\qquad\Rightarrow \qquad\qquad\h_tf\leq c\quad\mm\ae,\ \forall t>0.
\]
Finally, by Proposition 4.1 in \cite{DPRW09} the dimension-free Harnack inequality \eqref{eq:harnack} implies the strong Feller property for $\h_t$ for all $t>0$, namely if $f \in L^\infty(\X)$, then $\h_t f$ is continuous and bounded.

\medskip

Let us recall that the minimal weak upper gradient is local (i.e.\ $|Df| = |Dg|$ $\mm$-a.e.\ on $\{f=g\}$), lower semicontinuous w.r.t.\ $\mm$-a.e.\ convergence and that Lipschitz functions with bounded support are dense in $L^p(\X)$, $p \in (1,\infty)$ (see \cite{AmbrosioGigliSavare11-3}, where the density is actually proved in $W^{1,p}(\X)$).

As regards the properties of the differential, the following calculus rules (see \cite{Gigli14} for the proof) will be used extensively without further notice:
\begin{align*}
|\d f|&=|D f|\quad\mm\ae&&\forall f\in S^2(\X)\\
\d f&=\d g\qquad\mm\ae\ \text{\rm on}\ \{f=g\} &&\forall f,g\in S^2(\X)\\
\d(\varphi\circ f)&=\varphi'\circ f\,\d f&&\forall f\in S^2(\X),\ \varphi:\R\to \R\ \text{Lipschitz}\\
\d(fg)&=g\,\d f+f\,\d g&&\forall f,g\in L^\infty\cap S^2(\X)
\end{align*}
where it is part of the properties the fact that $\varphi\circ f,fg\in S^2(\X)$ for $\varphi,f,g$ as above.

Finally, given $f : \X \to \R$, the {\bf local Lipschitz constant} $\lip(f):\X\to[0,\infty]$ is defined as 0 on isolated points and otherwise as
\[
\lip f(x) := \limsup_{y \to x}\frac{|f(x) - f(y)|}{d(x,y)}.
\]
If $f$ is Lipschitz, then its Lipschitz constant is denoted by $\Lip(f)$. It is worth stressing that if $f$ is Lipschitz with $\lip(f) \in L^2(\X)$, then $f \in S^2(\X)$ with
\begin{equation}\label{eq:mwug-lip}
|D f| \leq \lip(f), \quad \mm\textrm-a.e.
\end{equation}

\begin{Remark}[Lipschitz cut-off function]\label{rem:cutoff}
{\rm
Given a complete and separable metric measure space $(\X,\sfd,\mm)$ equipped with a non-negative Borel measure which is finite on bounded sets, for any $x \in \X$ and $r>0$ there exists a Lipschitz function $\nchi_r : \X \to [0,1]$ such that $\nchi_r \equiv 1$ on $B_r(x)$, $\supp(\nchi_r) \subset B_{r+1}(x)$ and $\|\lip(\nchi_r)\|_{L^\infty(\X)}$ does not depend on $r$.

Indeed, if $\eta \in {\rm LIP}(\R)$ with bounded support, $\eta \equiv 1$ on $[0,1/3]$, $\eta \equiv 0$ on $[2/3,\infty)$ and set $\nchi_r := \eta \circ \sfd(\cdot,B_r(x))$, where $\sfd(\cdot,B_r(x)) := \inf_{y \in B_r(x)}\sfd(\cdot,y)$, then it is easy to see that $\nchi_r \equiv 1$ on $B_{r+1/3}(x)$, $\supp(\nchi_r) \subset B_{r+2/3}(x)$ and is Lipschitz with $\lip(\nchi_r) \leq |\eta'(\sfd(\cdot,B_r(x)))|\leq C$ $\mm$-a.e.\ with $C$ independent of $r$.
}\fr
\end{Remark}

\section{Auxiliary results}\label{sec:3}

In this section we collect all technical results that are required in view of the proof of Theorem \ref{thm:main}. We begin with a continuity statement for the heat semigroup.

\begin{Lemma}\label{lem:heatLp}
With the same assumptions and notations as in Setting \ref{set}, let $f \in L^1 \cap L^\infty(\X)$. Then $(\h_t f) \in C([0,\infty),L^p(\X))$ for all $p \in [1,\infty)$.

In particular, for all $T>0$, $p \in [1,\infty)$ and $\eps > 0$ there exists a bounded open set $B$ such that
\begin{equation}\label{eq:integrability}
\int_{\X \setminus B} |\h_t f|^p\,\d\mm < \eps, \qquad \forall t \in [0,T].
\end{equation}
\end{Lemma}

\begin{proof}
As a first step, recall that the heat flow satisfies the maximum principle and $(\h_t f) \in C([0,\infty),L^2(\X))$. Therefore, if $p \geq 2$ it is sufficient to observe that
\[
\int |\h_s f - \h_t f|^p\d\mm \leq \|\h_s f - \h_t f\|_{L^\infty(\X)}^{p-2}\int |\h_s f - \h_t f|^2\d\mm \leq 2\|f\|_{L^\infty(\X)}^{p-2} \|\h_s f - \h_t f\|_{L^2(\X)}^2
\]
for all $s,t \geq 0$ to deduce the $L^p$-continuity. For $p = 1$ we rely on Brezis-Lieb's lemma (see \cite{BrezisLieb83}); for our purposes it is enough to know that if $(u_n) \subset L^1(\X)$ is a bounded sequence with $u_n \to u$ $\mm$-a.e.\ for some measurable function $u$, then $u \in L^1(\X)$ and
\begin{equation}\label{eq:brezislieb}
\lim_{n \to \infty}\int\Big(|u_n| - |u_n - u|\Big)\d\mm = \int |u|\d\mm.
\end{equation}
Hence pick any $t \geq 0$, any sequence $(t_n)$ converging to it and set $u_n := \h_{t_n}f$. From $(\h_t f) \in C([0,\infty),L^2(\X))$ we know that, up to pass to a subsequence, $u_n \to u = \h_t f$ $\mm$-a.e.\ and from the mass-preserving property of the heat flow $\|u_n\|_{L^1(\X)} = \|u\|_{L^1(\X)}$. Therefore \eqref{eq:brezislieb} yields $\|u_n - u\|_{L^1(\X)} \to 0$ as $n \to \infty$ and by the arbitrariness of $(t_n)$ we conclude that $(\h_t f) \in C([0,\infty),L^1(\X))$. For $p \in (1,2)$, it is sufficient to notice that, by interpolation, there exists $C_p$ such that
\[
\|\h_s f - \h_t f\|_{L^p(\X)} \leq C_p\Big(\|\h_s f - \h_t f\|_{L^1(\X)} + \|\h_s f - \h_t f\|_{L^2(\X)}\Big), \qquad \forall s,t \geq 0.
\]
As regards \eqref{eq:integrability}, we first observe that if we set $u_n := |\h_{t_n}f|^p$ and $u := |\h_t f|^p$ with $t_n \to t$, then the $L^p$-continuity of the heat flow implies that, up to subsequences, $u_n \to u$ $\mm$-a.e.\ and $\|u_n\|_{L^1(\X)} \to \|u\|_{L^1(\X)}$ as $n \to \infty$. Plugging this information into \eqref{eq:brezislieb} we deduce that $(|\h_t f|^p) \in C([0,\infty),L^1(\X))$ and in turn this implies that $K_T := \{ |\h_t f|^p \,:\, t \in [0,T]\} \subset L^1(\X)$ is compact for all $T \geq 0$. Indeed, if $(u_n) \subset K_T$, then $u_n = |\h_{t_n}f|^p$ for some $t_n \in [0,T]$ and $(t_n)$ admits a convergent subsequence, say $t_{n_k} \to t$, so that $u_{n_k} \to u := |\h_t f|^p \in K_T$ in $L^1(\X)$. In particular $K_T$ is weakly compact in $L^1(\X)$ and, by the Dunford-Pettis theorem for $\sigma$-finite measures, this implies that for every $\eps > 0$ there exists a measurable set $B$ with finite measure such that
\[
\int_{\X \setminus B} |\h_t f|^p\,\d\mm < \eps, \qquad \forall t \in [0,T].
\]
In order to see that $B$ can be replaced by a bounded set, by the maximum principle and the fact that $f \in L^\infty(\X)$ we know that, for some constant $C>0$, $|\h_t f|^p \leq C$ for all $t \geq 0$ and, since $\mm(B) < \infty$, for all $x \in \X$ there exists $R>0$ such that $\mm(B \setminus B_R(x)) < \eps/C$. Therefore, fixing $x \in \X$ and setting $B' := B \cap B_R(x)$ we see that
\[
\int_{\X \cap B'} |\h_t f|^p\,\d\mm = \int_{\X \setminus B} |\h_t f|^p\,\d\mm + \int_{B \setminus B_R(x)} |\h_t f|^p\,\d\mm < 2\eps,
\]
whence the conclusion by the arbitrariness of $\eps$.
\end{proof}

Then we show that the dimension-free Harnack inequality implies the local logarithmic Sobolev inequality, which is the non-smooth counterpart of the main result contained in \cite{ATW09}. On smooth Riemannian manifolds \cite{Wang04} and in the setting of Bakry-\'Emery $\Gamma$-calculus \cite{BakryGentilLedoux14}, both \eqref{eq:harnack} and \eqref{eq:quasi-hamilton} are known to be equivalent to the curvature-dimension condition $\CD(K,\infty)$ (\cite{Sturm06I}, see also Section \ref{sec:4} for the definition), but in the present framework the implication \eqref{eq:harnack} $\Rightarrow$ \eqref{eq:quasi-hamilton} was still missing. As a byproduct we improve the strong Feller property of $\h_t$ to an $L^\infty$-${\rm LIP}$ regularization.

\begin{Proposition}\label{pro:loclogsob}
With the same assumptions and notations as in Setting \ref{set}, $\h_t$ maps $L^\infty(\X)$ into ${\rm LIP}(\X)$ and for any $f \in L^\infty$ positive, $t>0$ it holds
\begin{equation}\label{eq:linfty-lip}
\lip(\h_t f)^2 \leq \frac{2K}{e^{2Kt}-1}\h_t f\big(\h_t(f\log f) - \h_t f\log\h_t f\big) \qquad \textrm{pointwise.}
\end{equation}
In addition, for any $f \in L^p(\X)$ positive, $p \in (1,\infty)$, and for any $t>0$ it holds
\begin{equation}\label{eq:quasi-hamilton}
|D\h_t f|^2 \leq \frac{2K}{e^{2Kt}-1}\h_t f\big(\h_t(f\log f) - \h_t f\log\h_t f\big) \qquad \mm\textrm{-a.e.}
\end{equation}
\end{Proposition}

\begin{proof}
Let $x \in \X$, $y_n \to x$ and for sake of brevity set $\sfd_n := \sfd(x,y_n)$; let us also fix $\delta > 0$ and take $f \in L^\infty(\X)$ positive (not necessarily in $L^p(\X)$). Since the function $\Phi : (0,\infty) \times [0,\infty) \to \R$ defined by
\[
\Phi(z,\alpha) := \left\{\begin{array}{ll}
\displaystyle{\frac{z^\alpha-1}{\alpha}} & \qquad \text{ if } \alpha > 0\\
\log z & \qquad \text{ if } \alpha = 0
\end{array}\right.
\]
is continuous, by the weak Feller property of $\h_t$ the function $y \mapsto \Phi(\h_t f(y),\delta\sfd(x,y))$ is continuous as well, whence
\begin{equation}\label{eq:pizbuin}
\limsup_{n \to \infty}\frac{\h_t f(y_n) - \h_t f(x)}{\sfd_n} + \delta(\h_t f\log\h_t f)(x) = \limsup_{n \to \infty}\frac{(\h_t f)^{1+\delta\sfd_n}(y_n) - \h_t f(x)}{\sfd_n}.
\end{equation}
Moreover by \eqref{eq:harnack} we have
\[
\begin{split}
(\h_t & f)^{1+\delta\sfd_n}(y_n) - \h_t f(x) \leq \h_t|f|^{1+\delta\sfd_n}(x)\exp\Big(\frac{K(1+\delta\sfd_n)\sfd_n}{2\delta(e^{2Kt}-1)}\Big) - \h_t f(x) \\
& = \exp\Big(\frac{K(1+\delta\sfd_n)\sfd_n}{2\delta(e^{2Kt}-1)}\Big)\big(\h_t|f|^{1+\delta\sfd_n}(x) - \h_tf(x)\big) + \h_t f(x)\Big(\exp\Big(\frac{K(1+\delta\sfd_n)\sfd_n}{2\delta(e^{2Kt}-1)}\Big) - 1\Big)
\end{split}
\]
and on the one hand
\[
\begin{split}
\limsup_{n \to \infty}&\exp\Big(\frac{K(1+\delta\sfd_n)\sfd_n}{2\delta(e^{2Kt}-1)}\Big)\frac{\h_t|f|^{1+\delta\sfd_n}(x) - \h_tf(x)}{\sfd_n} = \limsup_{n \to \infty}\int\frac{|f|^{1+\delta\sfd_n} - f}{\sfd_n}\hr_t[x]\d\mm \\
& \leq \delta\int f\log f \hr_t[x]\d\mm = \delta\h_t(f\log f)(x),
\end{split}
\]
where the inequality is motivated by Fatou's lemma (indeed $(f\Phi(f,\delta\sfd_n))_n$ is bounded from above, since so is $f$), while on the other hand
\[
\limsup_{n \to \infty}\frac{\h_t f(x)}{\sfd_n}\Big(\exp\Big(\frac{K(1+\delta\sfd_n)\sfd_n}{2\delta(e^{2Kt}-1)}\Big) - 1\Big) = \frac{K}{2\delta(e^{2Kt}-1)}\h_t f(x).
\]
Plugging these observations into \eqref{eq:pizbuin} yields
\begin{equation}\label{eq:pizbuin2}
\limsup_{n \to \infty}\frac{\h_t f(y_n) - \h_t f(x)}{\sfd_n} \leq \delta\h_t(f\log f)(x) - \delta(\h_t f\log\h_t f)(x) + \frac{K}{2\delta(e^{2Kt}-1)}\h_t f(x).
\end{equation}
Next, assume that there exists $c>0$ such that $f \geq c$; since by smoothness $\Phi(z,\alpha)$ converges uniformly to $\log z$ on $[z_0,z_1]$ with $z_0 > 0$ as $\alpha \downarrow 0$, this implies that $f\Phi(f,\delta\sfd_n)$ converges uniformly to $f\log f$ as $n \to \infty$. As on the other hand $\hr[y_n] \rightharpoonup \hr[x]$ in $L^1(\X)$ by the strong Feller property of $\h_t$, this yields
\[
\lim_{n \to \infty}\frac{\h_t f(y_n) - \h_t|f|^{1+\delta\sfd_n}(y_n)}{\sfd_n} = -\delta\h_t(f\log f)(x),
\]
whence
\[
\limsup_{n \to \infty}\frac{\h_t f(x) - \h_t f(y_n)}{\sfd_n} - \delta\h_t(f\log f)(x) = \limsup_{n \to \infty}\frac{\h_t f(x) - \h_t|f|^{1+\delta\sfd_n}(y_n)}{\sfd_n}.
\]
Using again \eqref{eq:harnack} we obtain
\[
\begin{split}
\h_t & f(x) - \h_t|f|^{1+\delta\sfd_n}(y_n) \leq \h_t f(x) - (\h_t f)^{1+\delta\sfd_n}\exp\Big(-\frac{K(1+\delta\sfd_n)\sfd_n}{2\delta(e^{2Kt}-1)}\Big) \\
& \h_t f(x)\Big(1 - \exp\Big(-\frac{K(1+\delta\sfd_n)\sfd_n}{2\delta(e^{2Kt}-1)}\Big)\Big) + \exp\Big(-\frac{K(1+\delta\sfd_n)\sfd_n}{2\delta(e^{2Kt}-1)}\Big)\big(\h_t f(x) - \h_t|f|^{1+\delta\sfd_n}(x)\big)
\end{split}
\]
and on the one hand
\[
\limsup_{n \to \infty} \frac{\h_t f(x)}{\sfd_n}\Big(1 - \exp\Big(-\frac{K(1+\delta\sfd_n)\sfd_n}{2\delta(e^{2Kt}-1)}\Big)\Big) = \frac{K}{2\delta(e^{2Kt}-1)}\h_t f(x)
\]
while on the other hand
\[
\limsup_{n \to \infty}\exp\Big(-\frac{K(1+\delta\sfd_n)\sfd_n}{2\delta(e^{2Kt}-1)}\Big)\frac{\h_t f(x) - \h_t|f|^{1+\delta\sfd_n}(x)}{\sfd_n} = -\delta(\h_t f\log\h_t f)(x),
\]
so that
\begin{equation}\label{eq:pizbuin3}
\limsup_{n \to \infty}\frac{\h_t f(x) - \h_t f(y_n)}{\sfd_n} \leq \delta\h_t(f\log f)(x) + \frac{K}{2\delta(e^{2Kt}-1)}\h_t f(x) - \delta(\h_t f\log\h_t f)(x).
\end{equation}
The assumption $f \geq c$ for some $c>0$ can now be removed, since if we define $f_k := f + 1/k$ then \eqref{eq:pizbuin3} holds for $f_k$ and passing to the limit as $k \to \infty$ yields the validity of \eqref{eq:pizbuin3} for $f$: indeed, the left-hand side is easily seen to be constant w.r.t.\ $k$, while on the right-hand side it is sufficient to use the representation formula \eqref{eq:rapprform} to deduce that $\h_t f_k \to \h_t f$ and $\h_t(f_k\log f_k) \to \h_t(f\log f)$ pointwise.

Combining \eqref{eq:pizbuin3} with \eqref{eq:pizbuin2} and recalling the definition of $\lip(f)$ imply
\[
\lip(\h_t f) \leq \delta\h_t(f\log f) - \delta(\h_t f\log\h_t f) + \frac{K}{2\delta(e^{2Kt}-1)}\h_t f
\]
pointwise, as the right-hand side is continuous by the strong Feller property of $\h_t$, and it is now sufficient to optimize the right-hand side w.r.t.\ $\delta$ to get \eqref{eq:linfty-lip} and the $L^\infty$-${\rm LIP}$ regularization for $\h_t$. In order to prove \eqref{eq:quasi-hamilton} it is sufficient to consider \eqref{eq:linfty-lip} for $f \in {\rm LIP}(\X)$ positive with bounded support, recall \eqref{eq:mwug-lip} and use the density of Lipschitz functions with bounded support in $L^p(\X)$, $p \in (1,\infty)$, together with the lower semicontinuity of the minimal weak upper gradient.
\end{proof}


Finally we prove the integral maximum principle for the heat semigroup. On Riemannian manifolds several different proofs are possible (see for instance \cite{Grigoryan94} and references therein); here we adapt to the metric measure framework the one proposed in \cite{GongWang01}.

\begin{Lemma}[Integral maximum principle]\label{lem:contraction}
With the same assumptions and notations as in Setting \ref{set}, let $T>0,p \in (1,\infty),x \in \X$ and set
\[
\xi_t(y) := -\frac{\sfd^2(x,y)}{2(T-qt)}, \qquad \textrm{for }y \in \X,t \in [0,T/q),
\]
where $q := \frac{p}{2(p-1)}$. Then, for any non-negative $f \in L^p(\X)$ and $t \in [0,T/q)$, it holds
\[
\int (\h_t f)^p e^{\xi_t}\d\mm \leq \int f^p\exp\Big(-\frac{\sfd^2(x,\cdot)}{2T}\Big)\d\mm.
\]
\end{Lemma}

\begin{proof}
As a preliminary remark, by standard approximation argument it is not restrictive to assume that $f \in L^\infty(\X)$ with bounded support: by Lemma \ref{lem:heatLp} this implies that $(\h_t f) \in C([0,\infty),L^p(\X))$. Hence for all $t \geq 0$ and any sequence $(t_n)$ converging to $t$ there exists $g \in L^p(\X)$ such that, up to pass to a subsequence, $\h_{t_n}f \leq g$ for all $n \in \N$. As $e^{\xi_t} \leq 1$ and $\xi_t$ smoothly depends on $t \in [0,T/q)$, by dominated convergence we deduce that
\[
I(t) := \int (\h_t f)^p e^{\xi_t}\d\mm
\]
is continuous on $[0,T/q)$. In order to see that it is also locally absolutely continuous on $(0,T/q)$, let $R>0$ and take a cut-off function $\nchi_R$ as in Remark \ref{rem:cutoff}. Notice that $(\h_t f) \in AC_{loc}((0,\infty),L^2(\X))$ and, by the maximum principle, it is uniformly bounded in space and time. As $\xi_t$ smoothly depends on $t \in [0,T/q)$ and $\sfd^2(x,\cdot) \in W^{1,2}_{loc}(\X)$, we deduce that $(0,T/q) \ni t \mapsto \nchi_R (\h_t f)^p e^{\xi_t} \in L^2(\X)$ is absolutely continuous. In particular, so is $\int \nchi_R (\h_t f)^p e^{\xi_t}\d\mm$ and it is then clear that
\[
\ddt\int\nchi_R (\h_t f)^p e^{\xi_t}\d\mm = \int\nchi_R\Big(pe^{\xi_t}(\h_t f)^{p-1}\Delta\h_t f - \frac{q\sfd_x^2}{2(T-qt)^2}e^{\xi_t} (\h_t f)^p\Big)\d\mm, \qquad \textrm{a.e. } t
\]
where $\sfd_x := \sfd(x,\cdot)$. As regards the first summand on the right-hand side, using integration by parts and the fact that $|\nabla\sfd_x| \leq 1$ $\mm$-a.e.\ we can rewrite it as
\[
\begin{split}
\int\nchi_R e^{\xi_t}(\h_t f)^{p-1}\Delta\h_t f \d\mm & = -\int\langle\nabla(\nchi_R e^{\xi_t}(\h_t f)^{p-1}),\nabla \h_t f\rangle\d\mm \\ 
& = -\int\nchi_R\Big((p-1)(\h_t f)^{p-2}|\nabla\h_t f|^2 + (\h_t f)^{p-1}\langle\nabla\xi_t,\nabla\h_t f\rangle\Big)e^{\xi_t}\d\mm \\ 
& \qquad -\int e^{\xi_t}(\h_t f)^{p-1} \langle\nabla\nchi_R,\nabla\h_t f\rangle\d\mm \\
& \leq \int \nchi_R\Big(-(p-1)(\h_t f)^{p-2}|\nabla\h_t f|^2 + \frac{\sfd_x}{T-qt}(\h_t f)^{p-1} |\nabla\h_t f|\Big) e^{\xi_t}\d\mm \\
& \qquad -\int (\h_t f)^{p-1}\langle\nabla\nchi_R,\nabla\h_t f\rangle e^{\xi_t}\d\mm
\end{split}
\]
so that
\[
\begin{split}
\ddt\int\nchi_R (\h_t f)^p e^{\xi_t}\d\mm & \leq -p\int \Big((\h_t f)^{p-1}\langle\nabla\nchi_R,\nabla\h_t f\rangle + (p-1)\nchi_R (\h_t f)^{p-2}|\nabla\h_t f|^2\Big) e^{\xi_t}\d\mm \\ 
& \quad + \int \nchi_R\Big(\frac{p\sfd_x}{T-qt} (\h_t f)^{p-1}|\nabla\h_t f| - \frac{q\sfd_x^2}{2(T-qt)^2}(\h_t f)^p \Big) e^{\xi_t}\d\mm.
\end{split}
\]
The fact that $\h_t f \in L^p(\X)$ for all $t \geq 0$, $0 \leq \nchi_R \leq 1$ and $\nchi_R \to 1$ $\mm$-a.e.\ as $R \to \infty$ are sufficient to deduce by dominated convergence
\[
\lim_{R \to \infty} \int\nchi_R (\h_t f)^p e^{\xi_t}\d\mm = \int (\h_t f)^p e^{\xi_t}\d\mm, \qquad \forall t \in (0,T/q).
\]
Moreover, since $\sfd_x^2 e^{\xi_t} \in L^\infty(\X)$ uniformly in $t \in [0,T/q)$, from \eqref{eq:integrability} we see that for all $\eps > 0$ there exists $R$ sufficiently large such that
\[
\int (1-\nchi_R)(\h_t f)^p \sfd_x^2 e^{\xi_t}\d\mm < \eps, \qquad \forall t \in [0,T/q)
\]
whence
\[
\lim_{R \to \infty}\int \frac{\nchi_R\sfd_x^2}{(T-qt)^2} (\h_t f)^p e^{\xi_t}\d\mm = \int \frac{\sfd_x^2}{(T-qt)^2} (\h_t f)^p e^{\xi_t}\d\mm \qquad \textrm{loc. uniformly in } t \in [0,T/q)
\]
by the arbitrariness of $\eps$. Then \eqref{eq:quasi-hamilton} entails that for any $\mathcal{C} \subset (0,T/q)$ compact there exists $C>0$ such that $|\nabla\h_t f|^2 \leq C\h_t f(\h_t(f\log f) - \h_t f\log\h_t f)$ for all $t \in \mathcal{C}$, whence
\[
\begin{split}
(\h_t f)^{p-2}|\nabla\h_t f|^2 e^{\xi_t} & \leq C(\h_t f)^{p-1}\h_t(f\log f)e^{\xi_t} + C(\h_t f)^p|\log\h_t f|e^{\xi_t} \\
& C(\h_t f)^{p-1}\h_t(f\log f)e^{\xi_t} + C(\h_t f)^{\frac{p+1}{2}}(\h_t f)^{\frac{p-1}{2}}|\log\h_t f| e^{\xi_t} =: g_t
\end{split}
\]
On the one hand $(\h_t f)^{p-1}e^{\xi_t}, (\h_t f)^{(p-1)/2}|\log\h_t f| e^{\xi_t} \in L^\infty(\X)$ uniformly in $t \in \mathcal{C}$, since $z \mapsto z^\alpha|\log z|$ is bounded on $[0,\|f\|_{L^\infty(\X)}]$ for all $\alpha > 0$ and $(p-1)/2 > 0$; on the other hand \eqref{eq:integrability} applies to $\h_t(f\log f)$ and $(\h_t f)^{(p+1)/2}$, since for the former the fact that $f \in L^\infty(\X)$ with bounded support entails $f\log f \in L^1 \cap L^\infty(\X)$ while for the latter $(p+1)/2 \geq 1$. Thus, arguing as above, by \eqref{eq:integrability} we have that for all $\eps > 0$ there exists $R$ large enough so that
\[
\begin{split}
\int (1-\nchi_R)(\h_t f)^{p-2}|\nabla\h_t f|^2 e^{\xi_t}\d\mm & = \int_{\X \setminus B_R(x)}(1-\nchi_R)(\h_t f)^{p-2}|\nabla\h_t f|^2 e^{\xi_t}\d\mm \\
& \leq \int_{\X \setminus B_R(x)} g_t\,\d\mm < \eps, \qquad \forall t \in [0,T/q),
\end{split}
\]
whence
\[
\lim_{R \to \infty}\int\nchi_R (\h_t f)^{p-2}|\nabla\h_t f|^2 e^{\xi_t}\d\mm = \int (\h_t f)^{p-2}|\nabla\h_t f|^2 e^{\xi_t}\d\mm \qquad \textrm{loc. uniformly in } t \in (0,T/q).
\]
Finally, since $f\log f \leq f\log(f+2)$ and $f\log(f+2) > 0$ as $f > 0$, by \eqref{eq:quasi-hamilton} we have that for any $\mathcal{C} \subset (0,T/q)$ compact there exists $C>0$ such that
\[
\begin{split}
(\h_t f)^{p-1}|\nabla\h_t f| & \leq C(\h_t f)^{p-1/2}\sqrt{\h_t(f\log(f+2)) - \h_t f\log\h_t f} \\
& \leq C\frac{(\h_t f)^{p-1/2}}{\sqrt{\h_t(f\log(f+2))}}\h_t(f\log(f+2)) + C(\h_t f)^p|\log\h_t f| \\
& \leq \frac{C}{\log 2}(\h_t f)^{p-1}\h_t(f\log(f+2)) + C(\h_t f)^p|\log\h_t f|
\end{split}
\]
for all $t \in \mathcal{C}$, so that arguing as above, using the fact that $\sfd_x e^{\xi_t} \in L^\infty(\X)$ uniformly in $t \in [0,T/q)$ and $|\nabla\nchi_R| \in L^\infty(\X)$ uniformly in $R>0$ with $\nchi_R,|\nabla\nchi_R|$ converging $\mm$-a.e.\ to 0,1 respectively, we see that
\[
\begin{split}
& \lim_{R \to \infty} \int (\h_t f)^{p-1}\langle\nabla\nchi_R,\nabla\h_t f\rangle e^{\xi_t}\d\mm = 0, \\
& \lim_{R \to \infty} \int \nchi_R\frac{p\sfd_x}{T-qt} (\h_t f)^{p-1}|\nabla\h_t f| e^{\xi_t}\d\mm = \int \frac{p\sfd_x}{T-qt} (\h_t f)^{p-1}|\nabla\h_t f| e^{\xi_t}\d\mm,
\end{split}
\]
locally uniformly in $t \in (0,T/q)$. We thus obtain that $I \in AC_{loc}((0,T/q))$ with
\[
\begin{split}
I'(t) & \leq \int \Big(-p(p-1)(\h_t f)^{p-2}|\nabla\h_t f|^2 + \frac{p\sfd_x}{T-qt} (\h_t f)^{p-1}|\nabla\h_t f| - \frac{q\sfd_x^2}{2(T-qt)^2} (\h_t f)^p\Big) e^{\xi_t}\d\mm \\ 
& = - p(p-1)\int (\h_t f)^p\Big(\frac{|\nabla\h_t f|}{\h_t f} - \frac{\sfd_x}{2(p-1)(T-t)}\Big)^2 e^{\xi_t}\d\mm \leq 0
\end{split}
\]
and since $I$ is continuous up to $t=0$, this yields the conclusion.
\end{proof}

\section{Main result and applications}\label{sec:4}

We are now in the position to prove the upper Gaussian estimate for the heat kernel, the proof being partially inspired by \cite{Grigoryan94b} and \cite{GongWang01}.

\begin{Theorem}[Gaussian upper bound]\label{thm:main}
With the same assumptions and notations as in Setting \ref{set}, for any $\eps > 0$ there exist constants $C_\eps > 0$ and $C_K \geq 0$ such that, for every $x,y \in \X$ and $t > 0$, it holds
\begin{equation}\label{eq:upperbound}
\hr_t[x](y) \leq \frac{1}{\sqrt{\mm(B_{\sqrt{t}}(x))\mm(B_{\sqrt{t}}(y))}}\exp\Big( C_\eps(1+C_K t) - \frac{\sfd^2(x,y)}{(4+\eps)t}\Big).
\end{equation}
If $K \geq 0$, then $C_K$ can be chosen equal to 0.
\end{Theorem}

\begin{proof}
As a first step, for any $x \in \X$ and $t,D > 0$ define
\[
E_D(x,t) := \int (\hr_t[x](y))^2\exp\Big(\frac{\sfd^2(x,y)}{Dt}\Big)\d\mm(y) \in [0,\infty]
\]
and observe that by \eqref{eq:chapkol} and the triangle inequality
\[
\hr_t[x](y) = \exp\Big(-\frac{\sfd^2(x,y)}{2Dt}\Big) \int \hr_{t/2}[x](z)\exp\Big(\frac{\sfd^2(x,z)}{Dt}\Big)\hr_{t/2}[z](y)\exp\Big(\frac{\sfd^2(z,y)}{Dt}\Big)\d\mm(z),
\]
whence by the Cauchy-Schwarz inequality
\begin{equation}\label{eq:roughbound}
\hr_t[x](y) \leq \sqrt{E_D(x,t/2)E_D(y,t/2)}\exp\Big(-\frac{\sfd^2(x,y)}{2Dt}\Big).
\end{equation}
In order to estimate $E_D(x,t/2)$ and $E_D(y,t/2)$, let $f \in L^1 \cap L^\infty(\X)$ be non-negative, $T>0$ and $p \in (1,2)$, consider the Harnack inequality \eqref{eq:harnack} with exponent $2/p$, raise it to the power $p$, multiply both sides by $e^{\xi_t}$ with $\xi_t$ defined as in Lemma \ref{lem:contraction} and integrate w.r.t.\ $\mm$ in the $y$ variable to get
\[
\int (\h_t f(x))^2 e^{\xi_t(y)}\exp\Big(-\frac{pK\sfd^2(x,y)}{(2-p)(e^{2Kt}-1)}\Big)\d\mm(y) \leq \int (\h_t(f^{2/p})(y))^p e^{\xi_t(y)}\d\mm(y).
\]
On the one hand, as $f^{2/p} \in L^p(\X)$ Lemma \ref{lem:contraction} can be applied, whence
\[
\int (\h_t(f^{2/p}))^p e^{\xi_t}\d\mm \leq \int f^2 \exp\Big(-\frac{\sfd^2(x,\cdot)}{2T}\Big)\d\mm, \qquad \forall t \in [0,T/q)
\]
with $q$ also defined as in Lemma \ref{lem:contraction}. On the other hand
\[
\begin{split}
\int (\h_t f(x))^2 & e^{\xi_t(y)}\exp\Big(-\frac{pK\sfd^2(x,y)}{(2-p)(e^{2Kt}-1)}\Big)\d\mm(y) \\ 
& \geq (\h_t f(x))^2 \int_{B_{\sqrt{2t}}(x)} e^{\xi_t(y)}\exp\Big(-\frac{pK\sfd^2(x,y)}{(2-p)(e^{2Kt}-1)}\Big)\d\mm(y) \\
& \geq (\h_t f(x))^2 \exp\Big(-\frac{t}{T-qt} - \frac{2pKt}{(2-p)(e^{2Kt}-1)}\Big)\mm(B_{\sqrt{2t}}(x)).
\end{split}
\]
By combining these inequalities we obtain
\begin{equation}\label{eq:fn}
(\h_t f(x))^2 \leq \frac{1}{\mm(B_{\sqrt{2t}}(x))}\exp\Big(\frac{t}{T-qt} + \frac{2pKt}{(2-p)(e^{2Kt}-1)}\Big)\int f^2 \exp\Big(-\frac{\sfd^2(x,\cdot)}{2T}\Big)\d\mm
\end{equation}
for all $t \in (0,T/q)$ and in particular this is true when $f$ is equal to
\[
f_n := (n \wedge \hr_t[x])\exp\Big(n \wedge \frac{\sfd^2(x,\cdot)}{2T}\Big),
\]
where $a \wedge b := \min\{a,b\}$. By the representation formula \eqref{eq:rapprform} and by the monotone convergence theorem, as the $f_n$'s are monotonically increasing, we deduce that \eqref{eq:fn} actually holds with $f = \hr_t[x]e^{\sfd^2(x,\cdot)/2T}$ and this is equivalent to
\begin{equation}\label{eq:halfbound}
\int (\hr_t[x])^2 \exp\Big(\frac{\sfd^2(x,\cdot)}{2T}\Big)\d\mm \leq \frac{1}{\mm(B_{\sqrt{2t}}(x))}\exp\Big(\frac{t}{T-qt} + \frac{2pKt}{(2-p)(e^{2Kt}-1)}\Big)
\end{equation}
for all $t \in (0,T/q)$. Recalling that $q := \frac{p}{2(p-1)}$ and $p \in (1,2)$ is arbitrary, we observe that for all $D>2$ there exists $p'$ sufficiently close to 2 such that $D > 2q' > 2$ with $q' = \frac{p'}{2(p'-1)}$. Hence, since also $T$ is arbitrary, if we pick any $t > 0$ and set $T := Dt/2$, in such a way that $t = 2T/D < T/q'$, we deduce that \eqref{eq:halfbound} holds and
\[
\frac{t}{T-q't} = \frac{2}{D - 2q'} =: C_D,
\]
so that we have just shown that for all $D>2$ there exists $p' \in (1,2)$ and $C_D > 0$ both depending only on $D$ such that
\[
\int (\hr_t[x])^2 \exp\Big(\frac{\sfd^2(x,\cdot)}{Dt}\Big)\d\mm \leq \frac{1}{\mm(B_{\sqrt{2t}}(x))}\exp\Big(C_D + \frac{2p'Kt}{(2-p')(e^{2Kt}-1)}\Big), \qquad \forall t > 0.
\]
Now observe that
\[
\frac{2p'Kt}{(2-p')(e^{2Kt}-1)} \leq C'_D(1+C_K t), \qquad \forall t > 0
\]
for suitable constants $C'_D,C_K$ and it is easily seen that $C_K$ can be chosen equal to 0 when $K \geq 0$. Plugging this inequality into the previous one and recalling the definition of $E_D(x,t)$ it follows that for any $D>2$
\[
E_D(x,t) \leq \frac{1}{\mm(B_{\sqrt{2t}}(x))}\exp\Big(C''_D(1+C_K t)\Big), \qquad \forall t > 0
\]
for some $C''_D>0$ depending only on $D$ and combining this inequality with \eqref{eq:roughbound} the bound \eqref{eq:upperbound} follows.
\end{proof}

\begin{Remark}{\rm
In the case of smooth Riemannian manifolds the Varadhan asymptotic formula for short-time behaviour of the heat kernel states that
\[
\lim_{t \downarrow 0} t\log\hr_t[x](y) = -\frac{\sfd^2(x,y)}{4}, \qquad x \neq y.
\]
Hence \eqref{eq:upperbound} is sharp for short time.
}\fr
\end{Remark}

In the same spirit of the remark above we can formulate the following

\begin{Corollary}
With the same assumptions and notations as in Setting \ref{set}, if we further suppose that for any $x \in \X$ there exist constants $R_x,c_x > 0$, $0 < \alpha_x < 2$ such that
\[
\mm(B_r(x)) \geq c_x\exp(-1/r^{\alpha_x}), \qquad r < R_x,
\]
then the upper Varadhan estimate holds
\[
\lim_{t \downarrow 0} t\log\hr_t[x](y) \leq -\frac{\sfd^2(x,y)}{4}
\]
for all $x,y \in \X$, $x \neq y$.
\end{Corollary}

As a direct consequence of Proposition \ref{pro:loclogsob} and of our main theorem we obtain the local logarithmic Sobolev inequality and an upper Gaussian estimate for the heat kernel in $\RCD(K,\infty)$ spaces, which to date were both still missing. For reader's sake, let us recall that $(\X,\sfd,\mm)$ satisfies the $\RCD(K,\infty)$ condition (see \cite{AmbrosioGigliSavare11-2}) if it is infinitesimally Hilbertian and the relative entropy functional ${\rm Ent}_\mm : \prob\X \to \R \cup \{+\infty\}$ defined as
\[
\textrm{Ent}_\mm(\mu) := \left\{\begin{array}{ll}
\displaystyle{\int\rho\log(\rho)\,\d\mm} & \qquad \text{ if } \mu = \rho\mm \textrm{ and } (\rho\log(\rho))^- \in L^1(\X) \\
+\infty & \qquad \text{ otherwise}
\end{array}\right.
\]
is $K$-convex in $(\probt\X,W_2)$, namely for every $\mu,\nu \in \probt\X$ with ${\rm Ent}_\mm(\mu), {\rm Ent}_\mm(\nu) < \infty$ there exists a $W_2$-geodesic $(\mu_t)$ with $\mu_0 = \mu$, $\mu_1 = \nu$ and such that
\[
{\rm Ent}_\mm(\mu_t) \leq (1-t){\rm Ent}_\mm(\mu) + t{\rm Ent}_\mm(\nu) - \frac{K}{2}t(1-t)W_2^2(\mu,\nu), \qquad \forall t \in [0,1].
\]
In this class of spaces, H.\ Li proved that the dimension-free Harnack inequality \eqref{eq:harnack} holds \cite{Li16}, so that we can state the following

\begin{Corollary}
Let $(\X,\sfd,\mm)$ be an $\RCD(K,\infty)$ space with $K \in \R$. Then 
\begin{itemize}
\item[(i)] for any $f \in L^p(\X)$ positive, $p \in (1,\infty)$, and for any $t>0$ it holds
\[
\frac{e^{2Kt}-1}{2K} \frac{|\nabla\h_t f|^2}{\h_t f} \leq \h_t(f\log f) - \h_t f\log\h_t f \qquad \mm\textrm{-a.e.}
\]
\item[(ii)] for any $\eps > 0$ there exist constants $C_\eps > 0$ and $C_K \geq 0$ such that, for every $x,y \in \X$ and $t > 0$, it holds
\[
\hr_t[x](y) \leq \frac{1}{\sqrt{\mm(B_{\sqrt{t}}(x))\mm(B_{\sqrt{t}}(y))}}\exp\Big( C_\eps(1+C_K t) - \frac{\sfd^2(x,y)}{(4+\eps)t}\Big).
\]
If $K \geq 0$, then $C_K$ can be chosen equal to 0.
\end{itemize}
\end{Corollary}

\bibliographystyle{siam}
{\small
\bibliography{biblio}}

\end{document}